\newcommand{\commentt}[2]{#1}
\documentclass[a4paper,11pt]{article}
\usepackage{graphicx,slashbox,amsthm}
\usepackage[full]{harvard}
\usepackage[ps,dvips]{xy}
\title{Exact and efficient simulation of Gaussian vectors via PCA}
\title{Simulating Gaussian vectors via randomized dimension reduction and PCA}
\author{Nabil Kahal\'e
\thanks{\emph{ESCP Business School, 75011 Paris,
France; \commentt{e-mail: }{\email}{nkahale@escp.eu}.}}
}

\usepackage{amstext}
\usepackage{amssymb}
\usepackage{amsmath}
\usepackage{setspace}
\usepackage[margin=1in]{geometry}
\usepackage[english]{babel}
\usepackage{color}
\newcommand{\citet}{\citeasnoun}
\usepackage{algpseudocode}
\usepackage{algorithm}

\begin{document}

\newtheorem{example}{Example}[section]
\newtheorem{theorem}{Theorem}[section]
\newtheorem{conjecture}{Conjecture}[section]
\newtheorem{lemma}{Lemma}[section]
\newtheorem{proposition}{Proposition}[section]
\newtheorem{remark}{Remark}[section]
\newtheorem{corollary}{Corollary}[section]
\newtheorem{definition}{Definition}[section]
%\numberwithin{equation}{section}
%\numberwithin{proposition}{section}

%\newenvironment{proof}{{\bf Proof: }}{\vskip0.3in}%{\framebox[\totalheight]{\space }}}
\newcommand{\comment}[1]{}
\newcommand{\supp}[1]{\text{Supp}({#1})}
\newcommand{\vect}[1]{{\rm Vect}({\bf #1})}
\newcommand{\tr}{{\rm tr}}
\newcommand{\cov}{{\rm Cov}}
\newcommand{\E}{{\rm E}}
\newcommand{\var}{{\rm Variance}}
\newcommand{\diag}{{\rm diag}}
\newcommand{\std}{{\rm Std}}
 \maketitle
 %\begin{center}{Draft. Not for distribution.}\end{center}
\begin{abstract}
We study the problem of estimating \(E(g(X))\), where \(g\) is a real-valued function of \(d\) variables and \(X\) is a \(d\)-dimensional Gaussian vector with a given covariance matrix. We present a new unbiased estimator for  \(E(g(X))\) that   combines the randomized dimension reduction technique with principal components analysis. Under suitable conditions, we prove that our algorithm outperforms  the standard Monte Carlo method by a factor of order \(d\).
\end{abstract}
Keywords: Cholesky factorisation, Gaussian vectors, Principal components analysis, Monte Carlo simulation, variance reduction, dimension reduction
%{Subject classifications:} Finance: Asset pricing, Securities. Programming:
%Infinite dimensional. Mathematics: Convexity.
\section{Introduction}
Monte Carlo simulation of  Gaussian vectors  is commonly used in  a  variety of fields such
as
 weather and spatial prediction \cite{gel2004,diggle2007springer}, finance~\cite{glasserman2004Monte,Hull14}, and machine learning~\cite{russo2014learning,russoBVR2016}. Let \(M\) be a  \(d\times d\)   positive semi-definite matrix, and let \(X\) be a  \(d\)-dimensional
Gaussian vector with   mean \(0\) and covariance matrix \(M\).   A standard technique to simulate \(X\)  is Cholesky factorization \cite[Subsection 2.3.3]{glasserman2004Monte}, that uses a \(d\times d\) matrix  \(A\)    such that\begin{equation}\label{eq:aat}
AA^{T}=M.
\end{equation} 
Such a matrix  \(A\) \comment{is called a Cholesky factorization if  \(A\) is lower triangular. If \(V\) is positive definite, a Cholesky factorization exists and}can be computed in \(O(d^{{3}})\) time \cite[Subsections 4.2.5 and 4.2.8]{golub2013matrix}. More efficient  methods  simulate  Gaussian vectors 
in special cases. In particular, Fast Fourier transforms techniques exactly simulate Gaussian processes with a stationary covariance matrix  \cite{WoodChan94,dietrich1997fast}   and fractional Brownian surfaces \cite{stein2002fast}. Sparse Cholesky factorization~\cite{rue2001fast} and iterative techniques~\cite{aune2013iterative} efficiently generate
 Gaussian vectors with a sparse  precision matrix \(M^{-1}\).
For a general covariance matrix \(M\), \(X\) can be approximately simulated via a Markov chain algorithm \cite{kahaleGaussian2019} that   asymptotically converges to \(N(0,M)\) and does not require any precomputations. \citet{kahale2022unbiased} describes an unbiased version of the algorithm.

Variance reduction techniques that improve the performance of Monte Carlo simulation have been developed in various settings \cite{glasserman2004Monte,asmussenGlynn2007,giles2015multilevel}.
The quasi-Monte Carlo method (QMC), a technique related to Monte Carlo simulation, outperforms standard Monte Carlo in many applications  \cite[Chapter 5]{glasserman2004Monte}. Brownian bridges and principal components
analysis (PCA) often enhance the efficiency of QMC in the pricing of financial derivatives \cite[Subsection 5.5.2]{glasserman2004Monte}. The randomized dimension reduction method (RDR), a recent variance reduction technique, provably outperforms standard Monte Carlo in certain high-dimensional settings \cite{kahaRandomizedDimensionReduction20}. RDR estimates
\(E(f(U))\),   where the vector       \(U\) consists  of \(d\) independent  random variables, and  \(f\)  is a  function of \(d\) variables. When \(f\) does not depend equally on all its arguments, RDR  provides an unbiased estimator of  \(E(f(U))\) that, under suitable conditions,  outperforms standard Monte Carlo by a factor of order \(d\) \cite{kahaRandomizedDimensionReduction20}. The basic idea behind RDR is to simulate more often the important arguments of \(f\). RDR is related to previous variance reduction techniques such as the algorithm by 
\citet{glasserman2003resource}  that estimates the expectation of the sum of a sequence of random variables by stopping the sequence early.

This paper considers the problem of estimating \(E(g(X))\), where \(g\) is a real-valued function on \(\mathbb{R}^{d}\) such that \(g(X)\) is square-integrable. Let  \(A\) be a \(d\times d\) matrix satisfying \eqref{eq:aat}, and let \(U\)  be a \(d\)-dimensional vector of independent standard Gaussian random variables. Since \(AU\) has the same distribution as \(X\) \cite[Subsection 2.3.3]{glasserman2004Monte}, we have  \(E(g(X))=E(f(U))\), where \(f(z):=g(Az)\) for \(z\in\mathbb{R}^{d}\). In other words, \(E(g(X))\) is equal to the expectation of a functional of \(d\) independent standard Gaussian random variables. The standard Monte Carlo method estimates \(E(f(U))\)  
by taking the average of \(f\) over \(n\) independent copies of \(U\). Once \(A\) is calculated, \(AU\) can be simulated in \(O(d^{2})\) time. This paper uses  RDR and PCA  to provide more efficient estimates of  \(E(g(X))\). The basic idea behind our approach is to select a matrix  \(A\)   satisfying \eqref{eq:aat} so that the importance of the \(i\)-th argument of \(f\) decreases with \(i\). This allows RDR to estimate \(E(f(U))\) more efficiently than standard Monte Carlo. PCA provides a natural choice for such a matrix \(A\). Under certain conditions, we prove that, for the same computational cost, the variance of our estimator is lower than the variance of the standard Monte Carlo estimator by a factor of order \(d\).

PCA is widely used in a range of fields such as finance \cite{jamshidianZhu1996,glasserman2004Monte}, statistics, and machine learning \cite{hastieTibshirani2009elements,huang2022scaled}. Empirical analysis shows that the movements of several Gaussian vectors that arise in finance is largely explained by a small number of factors \cite{jamshidianZhu1996,Hull14}. Throughout the rest of the paper, the running time refers to the number of arithmetic
operations. \section{Reminder on RDR}
This section assumes that the components of the vector       \(U=(U_{1},\ldots,U_{d})\) are \(d\) independent  random variables taking values in a measurable space \(F\), that  \(f:F^{d}\rightarrow\mathbb{R}\)  is  Borel-measurable, and that \(f(U)\)  is square-integrable.  An informal assumption is that the importance of \(U_{i}\) decreases with \(i\). In certain applications, like Markov chains, the arguments of \(f\) are reordered to satisfy this assumption  \cite{kahaRandomizedDimensionReduction20}. The RDR algorithm generates copies of \(U\) by sampling more often the first components of \(U\). It  estimates
\(E(f(U))\) by taking the average of \(f\) over the  copies of \(U\).   More precisely, the generic  RDR   algorithm takes the function \(f\), a positive integer \(n\) and a vector \(q=(q_{0},\dots,q_{d-1})\in\mathbb{R}^{d}\) as parameters, with\begin{equation}\label{eq:condQ}
1=q_{0}\ge q_{1}\ge\cdots\ge q_{d-1}>0.
\end{equation}Let \((N_{k}\)), 
\(k\geq 1\),  be    a sequence of  independent random integers  in \(\{1,\dots,d\}\) such that \(\Pr(N_{k}>i)=q_{i}\) for 
\(k\geq 1\) and  \(i\in\{1,\dots,d\}\).
The algorithm simulates \(n\) copies \(V^{(1)},\dots,V^{(n)}\) of \(U\) along these steps:
\begin{enumerate}
\item 
First iteration.
Generate a copy \(V^{(1)}\) of \(U\) and calculate \(f(V^{(1)})\).

\item Loop. In iteration \(k+1\), where \(1\leq k\leq n-1\), let \(V^{(k+1)}\) be the vector obtained from  \(V^{(k)}\) by re-simulating the first \(N_{k}\) components of \(V^{(k)}\), without altering the remaining components.
Calculate \(f(V^{(k+1)})\).
\item Output \begin{displaymath}
f_{n}:=\frac{f(V^{(1)})+\dots+f(V^{(n)})}{n}.
\end{displaymath}
\end{enumerate}
As observed in \citet{kahaRandomizedDimensionReduction20}, we have 
\begin{equation}\label{eq:expectedfn}
E(f_{n})=E(f(U)).
\end{equation}
For  \(0\leq i\leq d\), let 
\begin{equation*}
C(i)=\cov(f(U),f(U')),
\end{equation*}
where  \(U'=(U'_{1},\ldots ,U'_{i},U_{i+1},\ldots,U_{d})\), and   \(U'_{1},\ldots ,U'_{i}\) are  random variables such that  \(U'_{j}\sim U_{j}\)  for \(1\leq j\leq i\), and   \(U'_{1},\ldots ,U'_{i}\),  \(U\) are independent. In other words, \(U\sim U'\), the first \(i\) components of \(U'\) are independent of \(U\), and the last \(d-i\) components of \(U\) and \(U'\) are the same. Note that \(C(0)=\var(f(U))\) and  \(C(d)=0\).   An intuitive argument  \cite{kahaRandomizedDimensionReduction20} shows that, if the last \(d-i\) arguments of \(f\) are not important,  then    \(C(i)\) should be small. Furthermore,  \citet{kahaRandomizedDimensionReduction20} proves  that  \(C(i)\) is a decreasing function of \(i\),  and that \begin{equation}\label{eq:lemmaVariance}
n\var(f_{n})\le2 \sum^{d-1}_{i=0}\frac{C(i)-C(i+1)}{q_{i}}.
\end{equation}
The expected running time of the RDR algorithm depends on \(q\) as well.   \citet{kahaRandomizedDimensionReduction20}  provides a numerical algorithm to compute a vector \(q\) that optimizes the asymptotic efficiency of the RDR algorithm.  Proposition~\ref{pr:upperBoundingCi}   gives a bound on  \(C(i)\) in terms of an approximation of \(f\)   by a function of its first \(i\) arguments.       
\begin{proposition}[\citet{kahaRandomizedDimensionReduction20}]\label{pr:upperBoundingCi}
For  \(1\leq i\leq d\), if   \(f_{i}:F^{i}\rightarrow\mathbb{R}\)  is  a Borel-measurable function such that \(f_{i}(U_1,\dots,U_i)\) is square-integrable, then 
\begin{equation*}
C(i)\le \var(f(U)-f_{i}(U_1,\dots,U_i)).
\end{equation*}
\end{proposition}

\section{An RDR algorithm for Gaussian simulation}\label{se:main}
Throughout the rest of the paper, assume that \(U=(U_{1},\ldots,U_{d})\)  is a \(d\)-dimensional vector of independent standard Gaussian random variables, that   \(M\) is a  \(d\times d\)   positive semi-definite matrix,  and that  \(A\) is a \(d\times d\) matrix satisfying \eqref{eq:aat}.  Let \(g\) be a real-valued function of \(d\) variables such that \(E(g(X)^{2})\) is finite, where \(X\sim N(0,M)\). We assume that the time to calculate \(g(z)\) is  \(O(d)\)  for \(z\in\mathbb{R}^{d}\).\(\)  Let \(\mu:=E(g(X))\), \(\Sigma:=\var(g(X))\), and  \(f(z):=g(Az)\) for \(z\in\mathbb{R}^{d}\). Thus, \(\mu=E(f(U))\) and  \(\Sigma=\var(f(U))\).  Our goal is to estimate \(\mu\).
\subsection{The algorithm description} Algorithm \ref{alg:RDRGaussian} gives a detailed implementation of the RDR algorithm applied to \(f(U)\). In this case, \(F=\mathbb{R}\). Algorithm \ref{alg:RDRGaussian} takes as parameters the function \(g\), the matrices \(M\) and \(A\), and a vector \(q\in\mathbb{R}^{d}\) satisfying \eqref{eq:condQ}. For simplicity, the number of iterations \(n\) is calculated in Step~2 of Algorithm \ref{alg:RDRGaussian} via the equation
\begin{equation}\label{eq:ndef}
n =\biggl\lceil\frac{d}{\sum_{i=0}^{d-1}q_{i}}\biggr\rceil.
\end{equation}  For \(k\geq 1\), \(I_{k}\) refers to the \(k\times k\) identity matrix.\begin{algorithm}[H]
\caption{The algorithm GRDR}
\label{alg:RDRGaussian}
\begin{algorithmic}[1]
\Function{GRDR}{$g,M, A,q$}
\State \(n \gets\biggl\lceil\frac{d}{\sum_{i=0}^{d-1}q_{i}}\biggr\rceil\)
\State Simulate \(U=(U_{1},\dots,U_{d})\sim N(0,I_{d})\)
\State \(X\gets AU\)
\State \(S\gets g(X)\)
\For{\(k\gets 1,n-1\)} 
\State
Simulate \(N\in\{1,\dots,d\}\) with  \(\Pr(N > i)=q_{i}\) for \(0\leq i\leq d-1\)
\State Simulate \((U'_{1},\dots,U'_{N})\sim N(0,I_{N})\)
\State \(U'\gets(U'_{1},\dots,U'_{N},U_{N+1},\dots,U_{d})\)
\State\(X\gets X+A(U'-U)\)
\State \(S\gets S+ g(X)\)
\State\(U\gets U'\)
\EndFor
\State\Return \(f_{n}\gets S/n\)
\EndFunction
\end{algorithmic}
\end{algorithm}
By \eqref{eq:expectedfn}, we have \(E(f_{n})=\mu\). We now estimate the expected running time \(T\) of Algorithm~\ref{alg:RDRGaussian}. Steps 2 through 5 are performed in \(O(d^{2})\) time. In each iteration of the loop, the last \(d-N\) components of \(U'-U\) are null. Thus, the calculation of \(A(U'-U)\) and Step 10 can be performed in \(O(dN)\) time. The running time of the remaining steps in the loop is \(O(d+N)\).  As \(N=\sum_{i=0}^{d-1}1\{N>i\}\), we have\begin{equation*}E(N)=\sum_{i=0}^{d-1}q_{i}.\end{equation*} Thus\begin{eqnarray*}T&\leq& c(d^{2}+(n-1)E(dN))\\
&=&c(d^{2}+(n-1)d(\sum_{i=0}^{d-1}q_{i})),
\end{eqnarray*} for some constant \(c\). By \eqref{eq:ndef}, it follows that \(T\leq 2cd^{2}\).
Hence, for a general matrix \(A\), the expected time to simulate \(f_{n}\) is of the same order of magnitude as the time to  simulate \(U\) and calculate \(g(AU)\).
\subsection{Standard Monte Carlo versus GRDR} 
Given \(\epsilon>0\), \(n'=\lceil \Sigma^{2}/\epsilon^{2}\rceil\) independent samples of \(X\sim N(0,M)\) are needed so that the variance of the average of \(n'\) independent copies of \(g(X)\) is at most \(\epsilon^{2}\). Calculating the Cholesky decomposition and simulating these \(n'\) copies takes \begin{equation*}
\tau_{MC}(\epsilon)=\Theta(d^{3}+\frac{\Sigma ^{2}}{\epsilon^{2}}d^{2})
\end{equation*}
time, as observed in \cite{kahaleGaussian2019}. Similarly, the expected time needed  to estimate \(\mu\) with precision \(\epsilon\) by averaging independent copies of  \(f_{n}\) is \begin{equation*}
\tau_{GRDR}(\epsilon)=\Theta(d^{3}+\frac{ \var(f_{n})}{\epsilon^{2}}d^{2}).
\end{equation*}
Thus, as \(\epsilon\) goes to \(0\), the performance of  GRDR versus standard Monte Carlo depends on the ratio  \( \var(f_{n})/\Sigma\). Combining \eqref{eq:lemmaVariance} and \eqref{eq:ndef} shows that  \begin{equation}\label{eq:VarianceBoundfn}
\var(f_{n})\le\frac{2}{d} \left(\sum_{i=0}^{d-1}q_{i}\right)\left(\sum^{d-1}_{i=0}\frac{C(i)-C(i+1)}{q_{i}}\right).
\end{equation}
Given \(M\) and \(g\), the RHS of  \eqref{eq:VarianceBoundfn} depends on the choice of \(A\) and \(q\). This is because the \(C(i)\)'s depend on \(f\) which, by  definition, depends on \(A\). Given \(q\), the RHS of  \eqref{eq:VarianceBoundfn}  is a linear combination of the \(C(i)\)'s with non-negative weights, and is therefore    an increasing function of \(C(i)\), for  \(0\leq i\leq d\). 

Motivated by \cite[Proposition 6]{kahaRandomizedDimensionReduction20}, we examine the case where \(q_{i}=1/(i+1)\) for \(0\leq i\leq d-1\). By \eqref{eq:VarianceBoundfn} and the inequality  \(\sum_{i=2}^{d}1/i\le\ln d\),
\begin{eqnarray}\label{eq:varfnharmonic}\var(f_{n})&\le&\frac{2(1+\ln d)}{d}\sum^{d-1}_{i=0}(i+1)(C(i)-C(i+1))\nonumber\\&=&\frac{2(1+\ln d)}{d}\sum^{d-1}_{i=0}C(i).
\end{eqnarray}
As \(C(0)=\Sigma\) and \(C(i)\) decreases with \(i\), it follows that \begin{displaymath}
\var(f_{n})\le2(1+\ln d)\Sigma.
\end{displaymath}In other words, the variance of \(f_{n}\) is always upper bounded by \(\Sigma\), up to a logarithmic factor. On the other hand, if \(C(i)\) decreases rapidly with \(i\),  \(\var(f_{n})\) can be much smaller than  \(\Sigma\). For instance, if \(C(i)\le c'\Sigma/i\), for some constant \(c'\) and \(1\leq i\leq d-1\), then \eqref{eq:varfnharmonic} implies that  \(\var(f_{n})\) is smaller than  \(\Sigma\) by a factor of order \(d/\ln^{2}(d)\). Under an additional assumption on \(g\), Section~\ref{se:BoundingCi} provides explicit bounds on the \(C(i)\)'s and on \(\var(f_{n})\) and studies another choice of \(q\). \section{Explicit bounds on \(\var(f_{n})\)}\label{se:BoundingCi}
For \(d\times d\) matrices \(B\) and \(B'\), we say that \(B\preccurlyeq B'\) if \(B'-B\) is positive semi-definite.  Let  \(\kappa>0\). Following  \citet{kahaleGaussian2019}, we say that \(g\) is \((\kappa,M)\)-Lipschitz
if, for any \((2d)\)-dimensional centered Gaussian vector \((Y,Y')\) with \(\cov(Y)\preccurlyeq M\) and \(\cov(Y')\preccurlyeq M\),  
\begin{displaymath}
E((g(Y)-g(Y'))^{2})\leq \kappa ^{2}E(||Y-Y'||^{2}).
\end{displaymath}  
In particular, any  \(\kappa \)-Lipschitz function is also  \((\kappa,M)\)-Lipschitz. Subsection~\ref{sub:Basket} gives an example of a  \((\kappa,M)\)-Lipschitz function that is not Lipschitz. This section assumes that   
\(g\)  is  \((\kappa,M)\)-Lipschitz and, for simplicity,  that \(M\) is positive definite. Lemma \ref{le:LipschitzBarC} below gives bounds on the \(C(i)\)'s under these assumptions.  
\begin{lemma}\label{le:LipschitzBarC}For \(0\leq i\leq d-1\), we have    \begin{equation*}
C(i)\leq\kappa ^{2} \sum ^{d}_{j=i+1}(A^{T}A)_{jj}.
\end{equation*}
 \end{lemma}\begin{proof}
Let \(P\) be the   \(d\times d\) diagonal matrix whose first \(i\) diagonal elements are \(1\) and  remaining entries are \(0\). Let  \(U'=PU\), \(Y=AU\) and \(Y'=AU'\). Then  \(Y\sim N(0,M)\). As    \(Y'=APU\), \((Y,Y')\) is a centered Gaussian vector and  \begin{eqnarray*}\cov(Y')&=&E(Y'Y'^{T})\\&=&APE(UU^{T})PA^{T}\\
&=&APA^{T}\\
&\preccurlyeq&AA^{T},\end{eqnarray*}where the third equation follows from the equalities \(E(UU^{T})=I_{d}\) and \(P^{2}=P\), and the last equation from the equality \(AA^{T}-APA^{T}=A(I_d-P)A^{T}\) and the inequality \(P\preccurlyeq I_{d}\). Thus \(\cov(Y')\preccurlyeq M\). Furthermore, \(Y-Y'=A\Gamma\), where \(\Gamma:=U-U'\). Hence \begin{eqnarray*}
C(i)
&\le&E((f(U)-f(U'))^{2})\\
&=&E((g(Y)-g(Y'))^{2})\\
&\le&\kappa^{2}E(||A\Gamma||^{2})\\
&=&\kappa^{2}E(\Gamma^{T}A^{T}A\Gamma)\\
&=&\kappa^{2}\sum ^{d}_{j=1}\sum ^{d}_{k=1}(A^{T}A)_{jk}E(\Gamma_{j}\Gamma_{k})\\
&=&\kappa^{2}\sum ^{d}_{j=i+1}(A^{T}A)_{jj}.
\end{eqnarray*}The first equation follows from Proposition \ref{pr:upperBoundingCi}, and the last one from the fact that \(\Gamma\) has a normal distribution with covariance matrix \(I_d-P\). \end{proof}

When \(i=0\), Lemma \ref{le:LipschitzBarC} implies that
\begin{eqnarray}\label{eq:Varg(X)Bound}
\Sigma&\leq&\kappa ^{2}\tr(A^{T}A)\nonumber\\
&=&\kappa ^{2}\tr(AA^{T})\nonumber\\
&=&\kappa ^{2}\tr(M).
\end{eqnarray}
\subsection{Choosing \(q\) in terms of \(A\)}
We first show  that, without loss of generality, we can assume that the sequence \((A^{T}A)_{ii}\), \(1\leq i\leq d\), is decreasing. Consider a permutation \(\pi\) of \(\{1,\dots,d\}\) such that the sequence \((A^{T}A)_{\pi(i)\pi(i)}\), \(1\leq i\leq d\), is decreasing. Let \(B\) be the \(d\times d\) matrix defined by \(B_{ij}=A_{i\pi(j)}\), for \(1\leq i,j\leq d\).   
Then a standard calculation shows that  \(BB^{T}=AA^{T}=M\) and \((B^{T}B)_{ii}=(A^{T}A)_{\pi(i)\pi(i)}\) for \(1\leq i\leq d\). Thus  \(B\) satisfies  \eqref{eq:aat} and the sequence \((B^{T}B)_{ii}\), \(1\leq i\leq d\), is decreasing. 

 Assume now that \(A\) satisfies  \eqref{eq:aat} and that the sequence \((A^{T}A)_{ii}\), \(1\leq i\leq d\), is decreasing. The discussion following   \eqref{eq:VarianceBoundfn} shows that   \eqref{eq:VarianceBoundfn} remains valid if the \(C(i)\)'s are replaced with their upper-bounds implied by Lemma \ref{le:LipschitzBarC}. Consequently,\begin{equation*}
\var(f_{n})\le\frac{2\kappa^{2}}{d}\biggl(\sum_{i=1}^{d}q_{i-1}\biggr) \sum^{d}_{i=1}\frac{(A^{T}A)_{ii}}{q_{i-1}}.
\end{equation*}
On the other hand, by    \eqref{eq:aat}, \(\det(A^{T}A)=\det M\).  Thus, as \(M\) is invertible, so is \(A^{T}A\). Consequently, \(A^{T}A\) is positive definite and \((A^{T}A)_{ii}>0\) for \(1\leq i \leq d\).   A standard calculation then shows that the RHS  is minimized when\begin{equation}\label{eq:optqGen}
q_{i-1} = \sqrt{\frac{(A^{T}A)_{ii}}{(A^{T}A)_{11}}},
\end{equation} for \(1\leq i\leq d\). Another justification for  \eqref{eq:optqGen} follows from \cite[Proposition 4]{kahaRandomizedDimensionReduction20}. For this choice of \(q\), we have   
\begin{equation}\label{eq:varfnkappaGen}
\var(f_{n})\le\frac{2\kappa^{2}}{d}\biggl(\sum_{i=1}^{d} \sqrt{(A^{T}A)_{ii}}\biggr)^{2}.
\end{equation}This yields Theorem~\ref{th:OptqGenA} below.
\begin{theorem}\label{th:OptqGenA}Assume that \(M\) is positive definite, that  \(A\) satisfies  \eqref{eq:aat}, that  the sequence \((A^{T}A)_{ii}\), \(1\leq i\leq d\), is decreasing, and that \(g\) is \((\kappa,M)\)-Lipschitz. If  \(q\) is given by \eqref{eq:optqGen} then \(T\leq 2cd^{2}\) for some constant \(c\) and \(\var(f_{n})\) satisfies \eqref{eq:varfnkappaGen}.\end{theorem}
By the Cauchy-Schwartz inequality, \begin{eqnarray*}
\biggl(\sum_{i=1}^{d} \sqrt{(A^{T}A)_{ii}}\biggr)^{2}&\le&d\sum_{i=1}^{d} (A^{T}A)_{ii}\\&=& d\tr(A^{T}A)\\&=& d\tr(M).
\end{eqnarray*}
Thus, the upper bound on  \(\var(f_{n})\) given by  \eqref{eq:varfnkappaGen} is, up to a multiplicative factor of \(2\), no worse than the upper bound on  \(\Sigma\) given by \eqref{eq:Varg(X)Bound}. On the other hand, if \((A^{T}A)_{ii}\) decreases rapidly with \(i\), the bound on  \(\var(f_{n})\) given by  \eqref{eq:varfnkappaGen} can be much smaller than the bound on  \(\Sigma\) given by \eqref{eq:Varg(X)Bound}. In particular, if \((A^{T}A)_{ii}\le c'i^{-2} (A^{T}A)_{11}\) for some constant \(c'\) independent of \(d\) and \(1 \le i\le d\), then the RHS of  \eqref{eq:varfnkappaGen} is smaller than that of  \eqref{eq:Varg(X)Bound} by a factor of order \(d/\ln^{2}d\). Similarly, if  \((A^{T}A)_{ii}\le c'i^{\gamma} (A^{T}A)_{11}\)  for some constants \(c'\) and \(\gamma<-2\) independent of \(d\) and \(1 \le i\le d\), then the RHS of  \eqref{eq:varfnkappaGen} is smaller than that of  \eqref{eq:Varg(X)Bound} by a factor of order \(d\). 
\subsection{Combining RDR and PCA}
    Let
\(\lambda_{1},\dots,\lambda_{d}\) be the eigenvalues of  \(M\) in decreasing order. Denote by  \(\Lambda\)  the   \(d\times d\) matrix diagonal matrix  whose \(i\)-th diagonal entry is \(\lambda_{i}\). Thus, there is an orthogonal \(d\times d\) matrix \(Q\) such that \(M=Q\Lambda Q^{T}\). The matrices \(Q\) and \(\Lambda\) can be  approximately computed in \(O(d^{3})\) time \cite[Chapter 8]{golub2013matrix}.   Theorem \ref{th:sqrtlamdai} combines the RDR method with PCA.  
  
\begin{theorem}\label{th:sqrtlamdai}
  Assume that 
\(g\)  is  \((\kappa,M)\)-Lipschitz. Set \(A=Q\sqrt{\Lambda}\) and \(q_{i}=\sqrt{\lambda_{i+1}/\lambda_{1}}\) for \(0\leq i\leq d-1\).   Then \(T\leq 2cd^{2}\) for some constant \(c\) and \begin{equation}\label{eq:lambda}
\var(f_{n})\leq\frac{2\kappa^{2}}{d}\left(\sum_{i=1}^{d}\sqrt{\lambda_{i}}\right)^{2}. 
\end{equation} 
\end{theorem}
\begin{proof}
  Note that  \(A\) satisfies  \eqref{eq:aat}. On the other hand, since \(A^{T}A=\Lambda\), we have  \((A^{T}A)_{ii}=\lambda_{i}\) for  \(1\leq i \leq d\). Thus the sequence \((A^{T}A)_{ii}\), \(1\leq i\leq d\), is decreasing. Applying Theorem~\ref{th:OptqGenA} concludes the proof.\end{proof}
The discussion following Theorem~\ref{th:OptqGenA} shows that, if \(\lambda_{i}\le c'i^{-2} \lambda_{1}\) for some constant \(c'\) independent of \(d\) and \(1 \le i\le d\), the RHS of  \eqref{eq:lambda} is smaller than that of  \eqref{eq:Varg(X)Bound} by a factor of order \(d/\ln^{2}(d)\).
Similarly, if  \(\lambda_{i}\le c'i^{\gamma} d\) for some constants \(c'\) and \(\gamma<-2\) independent of \(d\) and \(1 \le i\le d\), then the RHS of  \eqref{eq:lambda} is smaller than that of  \eqref{eq:Varg(X)Bound} by a factor of order \(d\).

\subsection{A Basket option}\label{sub:Basket}   Consider   \(d\) stocks \(S_{1},\dots,S_{d}\) that  follow the standard log-normal model~\cite[Subsection 3.2.3]{glasserman2004Monte} and pay no dividends. Assume that \(S_{1}(0)=S_{2}(0)=\dots=S_{d}(0)=1\), where \(S_{i}(t)\) denotes the price of \(S_{i}\) at time \(t\). Let \(r\) be the risk-free rate, and let \(\sigma_{i}\) denote the volatility of \(S_{i}\). Assume that the \(\sigma_{i}\)'s are bounded by a constant independent of \(d\).  A Basket call option with  strike \(K\) and maturity \(t\) is an option with a payoff of
\(((S_{1}(t)+\cdots+ S_{d}(t))/d-K)^{+}\)
at time \(t\). The price of the option is equal to \(E(g(X))\),
where, for \(x=(x_{1},\dots,x_{d})\in\mathbb{R}^{d}\), \begin{equation*}
g(x)=(\frac{1}{d}\sum^{d}_{i=1}\exp(-\frac{\sigma_{i} ^{2}}{2}t+\sigma_{i}\sqrt{t} x_{i})-Ke^{-rT})^{+}.
\end{equation*}
The covariance matrix \(M\) of \(X\) is given by \(M_{ij}=\text{Correl}(\ln(S_{i}(t)),\ln(S_{j}(t)))\)
for \(1\leq i\leq j\leq d\). \citet{kahaleGaussian2019}  studies this example in the context of approximate simulation of Gaussian vectors without using matrix decomposition and shows that \(\kappa=O(d^{-1/2})\), where the constant behind the \(O\) notation does not depend on \(d\).
Since \(\sum_{i=1}^d \lambda_{i}=\tr(M)=d\), we have  \(\lambda_{1}\leq d.\)   If \(\lambda_{i}\le c'i^{-2} \lambda_{1}\) for some constant \(c'\) independent of \(d\) and \(1 \le i\le d\), then   \eqref{eq:lambda} implies that \(\var(f_{n})=O(\ln^{2}(d)/d)\), and \(
\tau_{GRDR}(\epsilon)=\Theta(d^{3}+{\ln^{2}(d)d}{\epsilon^{-2}})\) for \(\epsilon>0\).
Similarly, if  \(\lambda_{i}\le c'i^{\gamma} \lambda_{1}\) for some constants \(c'\) and \(\gamma<-2\) independent of \(d\) and \(1 \le i\le d\), then   \(\var(f_{n})=O(1/d)\), and \(
\tau_{GRDR}(\epsilon)=\Theta(d^{3}+{d}{\epsilon^{-2}})\). In contrast, if \(\Sigma\) is of order \(1\), then  \(
\tau_{MC}(\epsilon)=\Theta(d^{3}+{d^{2}}{\epsilon^{-2}})\).
%\bibliographystyle{nature}
%\bibliographystyle{dcu}
%\bibliographystyle{cparalleles}
%\bibliographystyle{siam}
%\bibliographystyle{plain}
%\bibliographystyle{agsm}
%\bibliographystyle{jmr}
%\bibliographystyle{jPhysicsB}
%\bibliographystyle{kluwer}
%\bibliographystyle{nederlands}
%\renewcommand{\harvardand}{AND}
%\harvardparenthesis{none}
\bibliography{poly}
\end{document}